\documentclass[11pt]{amsart}

\usepackage{amsthm, amsfonts, amssymb, amscd}
\usepackage{tikz-cd}
\usepackage[pagebackref,colorlinks]{hyperref}

\usepackage{lipsum}

\newcommand\blfootnote[1]{%
  \begingroup
  \renewcommand\thefootnote{}\footnote{#1}%
  \addtocounter{footnote}{-1}%
  \endgroup
}

\theoremstyle{definition}
\newtheorem{ntn}{Notation}[section]

\theoremstyle{plain}
\newtheorem{cor}[ntn]{Corollary}
\newtheorem{lem}[ntn]{Lemma}
\newtheorem{prp}[ntn]{Proposition}
\newtheorem{thm}[ntn]{Theorem}

\theoremstyle{definition}
\newtheorem{rem}[ntn]{Remark}
\newtheorem{exa}[ntn]{Example}

\numberwithin{equation}{section}

\newcommand{\z}{\mathbb{Z}}
\newcommand{\q}{\mathbb{Q}}

\newcommand{\A}{\mathcal{A}}

\newcommand{\E}{\mathcal{E}}
\newcommand{\RR}{\mathcal{R}}

\newcommand{\mt}{\mapsto}
\newcommand{\lan}{\langle}
\newcommand{\ran}{\rangle}
\newcommand{\arr}{\rightarrow}
\newcommand{\larr}{\longrightarrow}

\newcommand{\se}{\subseteq}
\newcommand{\two}{\twoheadrightarrow}
\newcommand{\tail}{\rightarrowtail}

\newcommand{\inc}{{\rm inc}}
\newcommand{\coker}{{\rm coker}}
\newcommand{\im}{{\rm im}}
\newcommand{\Hom}{{\rm Hom}}

\newcommand{\id}{{\rm id}}

\newcommand{\GL}{{\rm GL}}
\newcommand{\tors}{{{\rm Tor}_1^{\z}}}
\newcommand{\exts}{{{\rm Ext}_{\z}^1}}

\newcommand{\OO}{\mathrm{O}}
\newcommand{\EO}{\mathrm{EO}}
\newcommand{\U}{\mathrm{U}}
\newcommand{\EU}{\mathrm{EU}}

\newcommand{\Spp}{\mathrm{Sp}}

\newcommand{\ep}{\epsilon}

\newtheoremstyle{athm}
  {}
  {}
  {\itshape}
  {}
  {\scshape}
  {}
  {.5em}
  {\thmnote{#3}}
\theoremstyle{athm}

\begin{document}
\thispagestyle{empty}

\newpage
\title{Third homology of perfect central extensions}
\author{B.  Mirzaii}
\author{F. Y. Mokari}
\author{D. C. Ordinola}
\begin{abstract}
For a central perfect extension of groups $A \tail G \two Q$, first we study the natural image of 
$H_3(A,\z)$ in $H_3(G, \z)$. As a particular case, we show that if the extension is universal this image
is 2-torsion. Moreover when the plus-construction of the classifying space of $Q$
is an $H$-space, we also study  the kernel of the surjective homomorphism 
$H_3(G,\z) \arr H_3(Q, \z)$. 
\end{abstract}
\maketitle

\section*{Introduction}
\blfootnote{{\sf 2020 Mathematics Subject Classification:} 20J06, 55P20, 19D55}
\blfootnote{{\sf Affiliation:} 
Instituto de Ci\^encias Matem\'aticas e de Computa\c{c}\~ao (ICMC),

Universidade de S\~ao Paulo (USP), S\~ao Carlos,  S\~ao Paulo, Brazil.

E-mails:\ bmirzaii@icmc.usp.br, 

\ \ \ \ \ \ \ \ \ \ \   f.mokari61@gmail.com, 

\ \ \ \ \ \ \ \ \ \ \  d.carbajalordinola@gmail.com}

Homologies and cohomologies  are important invariants that one can assign to a given group. Usually these (co)homology groups 
are too complicated to be computed explicitly. Therefore in many cases  results allowing to 
compare the homologies of different groups become quite important.

In this article, we  study such homomorphism for the third homology groups of a perfect central extension.
A central extension 
\[
A \tail G \two Q
\]
is called perfect if $G$ is a perfect
group, i.e. if $G=[G,G]$. The aim of the current paper is to study the natural
maps $H_3(A,\z) \arr H_3(G,\z)$ and $H_3(G,\z) \arr H_3(Q,\z)$ for such extensions.

The interest to this problem comes from two sources. First from algebraic  and Hermitian $K$-theory,
where various type of universal central extensions appear  \cite[\S 5]{milnor1971}, 
\cite[Chap. 4]{rosenberg1996}, \cite[\S 1.4, \S 5.5]{hahn-omeara1989}.  Second from algebraic topology and homology of groups that often one has 
to deal with different types of spectral sequences that usually are difficult to deal with \cite[Chap. VII]{brown1994}.

In Section \ref{central-ex}, we give a quick overview of the (lower) homologies of a central extension and discuss a result of
Eckmann and Hilton.

In Section \ref{wqf} we study Whitehead's quadratic functor which plays an important role in 
this article.

In Section \ref{stem} we show that if $A$ is a central subgroup of a group $G$ 
such that $A\subseteq G'$, e.g. $G$ a perfect group, then the
image of the natural map
\begin{equation}\label{2-torsion}
H_3(A,\z)\arr H_3(G,\z)/\rho_\ast(A\otimes_\z H_2(G,\z))
\end{equation}
in 2-torsion, where $\rho: A \times G \arr G$ is the usual product map. In particular,
if $A\tail G\two Q$ is a universal central extension, then the
image of $H_3(A,\z)$ in $H_3(G,\z)$ is 2-torsion.

Section \ref{h-group} has $K$-theoretic flavour. We show that
if $A \tail G \two Q$ is a perfect central extension such that $K(Q,1)^+$, the 
plus-construction of the classifying space of $Q$, is an $H$-space, then we have
the exact sequence
\begin{equation}\label{exact-A/2}
A/2 \arr H_3(G,\z) /\rho_\ast(A\otimes_\z H_2(G,\z)) \arr H_3(Q,\z) \arr 0.
\end{equation}
Moreover we prove that with this extra condition, the map (\ref{2-torsion}) is trivial. In particular,
if the extension is universal,  then the image of $H_3(A,\z)$ in $H_3(G,\z)$ is trivial.

In Section \ref{k-th}, we study the third Hermitian $K$-group of a ring. Let $R$ be a ring with involution and a  central element 
$\epsilon $ such that $\epsilon\overline{\epsilon}=1$. Using the exact sequence (\ref{exact-A/2}) we show that we have the exact sequence
\[
{}_\epsilon K_2^h(R)/2\arr {}_\epsilon K_3^h(R)\arr H_3({}_\epsilon \EO(R),\z)\arr 0,
\]
where ${}_\epsilon \EO(R)$ is the elementary subgroup of the stable ortogonal group ${}_\epsilon \OO(R)$ of the pair $(R,\epsilon)$.

Finally in Section \ref{coho} we prove a cohomological version of the exact sequence~(\ref{exact-A/2}).  In fact we show that if 
$A \tail G \two Q$ is a perfect central extension such that $K(Q,1)^+$ is an  $H$-space, then we have the exact sequence
\[
0\arr \exts(A,\z) \arr H^3(Q,\z) \arr H^3(G,\z) \overset{ \rho^\ast}\larr (A\otimes_\z H_2(G,\z))^\ast,
\]
where for an abelian group $M$, $M^\ast$ is the dual group $\Hom_\z(M,\z)$. In particular, if the extension is universal, then we have the
exact sequence
\[
0\arr \exts(A,\z) \arr H^3(Q,\z) \arr H^3(G,\z) \arr 0.
\]

\subsection*{Notations}
We denote the commutator subgroup of a group $G$, by $G'$ or $[G, G]$.
If $A \arr B$ is a homomorphism of abelian groups, by $B/A$ we mean $\coker(A \arr B)$
and by $\im(A)$ we mean the image of $A$ in $B$.

\subsection*{Acknowledgements} F. Y. Mokari  acknowledges that during the work on this paper she was supported by a post-doc fellowship 
of FAPESP (Funda\c{c}\~ao de Amparo \`a Pesquisa do Estado de S\~ao Paulo) with grant number 2016-13937-9 and D. C. Ordinola 
acknowledges  that he was supported partly by CNPq (Conselho Nacional de Desenvolvimento Cient\'ifico e Tecnol\'ogico) and
partly by CAPES (Coordena\c{c}\~ao de Aperfei\c{c}oamento de Pessoal de N\'ivel Superior) PhD fellowships.

\section{The homology of central extensions}\label{central-ex}

Let  $A \tail G \two Q$ be a central extension. 
Standard classifying space theory gives a (homotopy theoretic) fibration \cite[Chap. 6]{davis-kirk2001}
of Eilenberg-MacLane spaces
\[
K(A,1) \arr K(G,1) \arr K(Q,1)
\]
(see  \cite[5.1.28]{rosenberg1996}). From this we obtain the fibration 
\[
K(G,1) \arr K(Q,1) \arr K(A,2)
\]
(see \cite[Lemma 3.4.2]{may-ponto2012}). By studying the Serre spectral sequences  
associated to the morphism of fibrations 
\[
\begin{tikzcd}
\Omega K(A,2)\ar[r] \ar[d] & PK(A,2) \ar[r] \ar[d] & K(A,2) \ar[d]\\
 K(G,1) \ar[r] & K(Q,1) \ar[r] & K(A,2),
\end{tikzcd}
\]
\cite[Chap.~9]{davis-kirk2001},
the first row being the path space fibration over  $K(A,2)$ \cite[Section 6.4]{davis-kirk2001}, Eckmann and Hilton  proved
the following theorem.

\begin{thm}$($\cite[Theorem 1.1]{eckmann-hilton1971}$)$\label{E-H}
For any central extension $A \tail G \two Q$, there is a natural map $\tau: H_4(K(A,2),\z)\arr A \otimes_\z H_1(G,\z)$
and a natural exact sequence
\[
\begin{array}{c}
H_4(Q,\z)\arr\ker(\tau) \arr H_3(G,\z)/\rho_\ast\Big(A\otimes H_2(G,\z)\oplus\tors(A,H_1(G,\z))\Big)\\
\arr H_3(Q,\z)\arr \coker(\tau) \arr H_2(G,\z) \arr H_2(Q,\z) \\
\arr A \arr H_1(G,\z) \arr H_1(Q,\z) \arr 0,
\end{array}
\]
where $\rho: A \times G \arr G$ is the product map $(a,g) \mapsto ag$. 
\end{thm}

\begin{cor}\label{E-H-1}
If $A \tail G \two Q$ is a central perfect extension, then we have the exact sequence
\[
H_4(Q, \z)\!\arr \!H_4(K(A,2),\z) \!\arr \!
H_3(G, \z)/\rho_\ast\Big(A\otimes_\z H_2(G,\z)\Big)\!\arr\! H_3(Q,\z)\!\arr\! 0.
\]
\end{cor}

The group $H_4(K(A,2),\z)$ plays a very important role in this article.
It has interesting properties and has been studied extensively in
\cite[Chap. 2]{whitehead1950} and \cite[Section~13]{eilenberg-maclane1954}.
It is closely connected to Whitehead's quadratic functor, which is the topic of the next section.

\section{Whitehead's quadratic functor}\label{wqf}

A function $\theta: A \arr B$ of (additive) abelian groups is called a quadratic map if
\par (1) for any $a \in A$, $\theta(a)=\theta(-a)$,
\par (2) the function $A \times A \arr B$ with
$(a,b) \mapsto \theta(a+b)-\theta(a)-\theta(b)$ is bilinear.
\medskip

For any abelian group $A$, there is a universal quadratic map 
\[
\gamma: A \arr \Gamma(A)
\]
such that for any quadratic map $\theta: A \arr B$, there is a unique group homomorphism $\Theta: \Gamma(A) \arr B$ such that
\[
\Theta\circ \gamma=\theta.
\] 
It is easy to see that $\Gamma$ is a functor from the category of abelian groups to itself.

The functions 
\[
\phi: A \arr A/2, \ \  \phi(a)=\overline{a}, \ \ {\rm and} \ \ \ \psi: A \arr A \otimes_\z A,   \ \ \psi(a)=a\otimes a,
\]
are quadratic maps. Thus we get the canonical homomorphisms
\[
\Phi: \Gamma(A) \arr A/2,\  \gamma(a) \mapsto \overline{a}
\ \ \ \  \text{and} \ \ \ \  
\Psi:\Gamma(A) \arr A\otimes_\z A,\ \gamma(a) \mapsto a\otimes a.
\] 
Clearly $\Phi$ is surjective. Moreover $\coker(\Psi)=A \wedge A\simeq H_2(A,\z)$ \cite[Theorem~6.4(iii),~Chap.~V]{brown1994}
and hence we have the exact sequence
\begin{equation}\label{exact-Gamma1}
\Gamma(A)\overset{\Psi}{\larr} A\otimes_\z A\arr H_2(A,\z)\arr 0.
\end{equation}

Furthermore we have the bilinear pairing 
\[
\Omega: A \otimes_\z A \arr \Gamma(A), 
\ \ a\otimes b \mt [a,b]:=\gamma(a+b)-\gamma(a)-\gamma(b).
\]
It is easy to see that for any $a,b,c \in A$, 
\par (i) $[a,b]=[b,a]$, 
\par (ii) $[a+b,c]=[a,c]+[b,c]$,
\par (iii) $\Phi([a,b])=0$, 
\par (iv) $\Psi([a,b])=a\otimes b + b \otimes a$.\\
~\\
Using (1) and (ii), for any $a,b,c \in A$, we obtain
\par (a) $\gamma(a)=\gamma(-a)$,
\par (b) $\gamma(a+b+c)-\gamma(a+b)-\gamma(a+c)-\gamma(b+c)
\!+\!\gamma(a)\!+\!\gamma(b)\!+\!\gamma(c)=0$.\\
\smallskip

Using these properties we can construct $\Gamma(A)$. 
Let $\A$ be the free abelian group generated by the symbols $w(a)$, $a \in A$. Set 
\[
\Gamma(A):=\A/\RR,
\]
where $\RR$ is the subgroup generated by the elements 
\par - $w(a)-w(-a)$ and 
\par - $w(a+b+c)-w(a+b)-w(a+c)-w(b+c)+w(a)+w(b)+w(c)$,\\
with $a,b,c \in A$. Now
\[
\gamma:A \arr \Gamma(A)
\] 
is given by
$a \mapsto \overline{w(a)}$.  It is easy to show that 
\[
[a,a]=2\gamma(a).
\]
Thus the composite 
\[
\Gamma(A) \overset{\Psi}{\arr} A \otimes_\z A \overset{\Omega}{\larr} \Gamma(A) 
\]
coincides with multiplication by $2$. This implies that the kernel of $\Psi$ is $2$-torsion.
Moreover, one sees easily that the composite
\[
A \otimes_\z A \overset{\Omega}{\larr} \Gamma(A)  \overset{\Psi}{\arr} A \otimes_\z A
\]
sends $a\otimes b$ to $a\otimes b+ b\otimes a$. It is known that the sequence  
\begin{equation}\label{exact-Gamma2}
A \otimes_\z A \overset{\Omega}{\larr} \Gamma(A) \overset{\Phi}{\arr} A/2 \arr 0
\end{equation}
is exact.

\begin{prp}\label{KA2}
For any abelian group $A$,  $\Gamma(A)\simeq H_4(K(A,2),\z)$.
\end{prp}
\begin{proof}
See \cite[Theorem 21.1]{eilenberg-maclane1954}
\end{proof}

Now a topological proof of the exact sequence (\ref{exact-Gamma1}) can be obtained by applying Theorem \ref{E-H} to the
central extension 
\[
A\overset{\simeq}{\tail} A \two \{1\}. 
\]
A topological proof of the exact sequence (\ref{exact-Gamma2}) may be obtained by studying the Serre 
spectral sequence associated to the path space fibration 
\[
\Omega K(A,2)\arr PK(A,2) \arr K(A,2).
\]
We refer the reader to \cite[pages 349--350]{eckmann-hilton1971} for analysis of the Serre spectral sequence associated
to a fibration $F \arr E \arr B$ such that $F$ is $0$-connected, $B$ is 1-connected and $H_3(B,\z)=0$. Observe 
that 
\[
\Omega K(A,n)=K(A,n+1)
\]
and for $n\geq 3$
\[
H_{n+2}(K(A, n),\z) \simeq A/2
\]
\cite[Theorem 3.20, Chap. XII]{whitehead1978}.
We should mention  that $K(A,2)$ is an $H$-space \cite[Theorem 7.11, Chap. V]{whitehead1978} 
and the pairing
\[
A \otimes_\z A\arr H_4(K(A,2),\z)
\]
is induced by the $H$-space structure of $K(A,2)$ (see \cite[page 349]{eckmann-hilton1971}).

\section{Third homology of central subgroups}\label{stem}

Let $A$ be a central subgroup of $G$ such that $A\subseteq G'=[G,G]$. The condition $A\se G'$ is equivalent 
to the triviality of the homomorphism of homology groups $H_1(A,\z)\arr H_1(G,\z)$. Let $n$ be a positive integer.
From the commutative diagram 
\begin{equation}\label{prod}
\begin{tikzcd}
A \times A \ar[d] \ar[r, "\mu"] & A \ar[d]\\
A \times G \ar[r, "\rho"] & G,
\end{tikzcd}
\end{equation}
where $\mu$ and $\rho$ are the usual product maps, we obtain the commutative diagram
\[
\begin{tikzcd}
H_{n-1}(A,\z) \otimes_\z H_1(A,\z) \ar[d, "=0"] \ar[r] & H_n(A,\z) \ar[d]\\
H_{n-1}(A,\z) \otimes_\z H_1(G,\z) \ar[r] & H_n(G,\z).
\end{tikzcd}
\]
The Pontryagin product (see \cite[Chap. V, \S 5]{brown1994}) induces a natural map 
\[
\begin{array}{c}
\bigwedge_\z^nA \arr H_n(A,\z)
\end{array}
\]
which always is injective \cite[Theorem 6.4(i), Chap.~V]{brown1994}. The above diagram shows that the composite 
\[
\begin{array}{c}
\bigwedge_\z^nA \arr H_n(A,\z) \arr H_n(G,\z)
\end{array}
\]
is trivial. Since the natural map 
\[
\begin{array}{c}
\bigwedge_\q^n(A\otimes_\z \q) \arr H_n(A,\q)
\end{array}
\]
is an isomorphism \cite[Theorem 6.4(ii), Chap. V]{brown1994}, the group $H_n(A,\z)/\bigwedge_\z^nA$ is torsion. 
This implies that the image of $H_n(A,\z)$ in $H_n(G,\z)$ is  torsion. Observe that since $H_2(A,\z)\simeq A\wedge A$ 
\cite[Theorem 6.4(iii), Chap.~V]{brown1994}, the  map
\[
H_2(A,\z)\arr H_2(G,\z)
\]
becomes trivial. 

Our first main result concerns the map $H_3(A,\z)\arr H_3(G,\z)$. For the study of this map we need the 
following well-known result.

\begin{prp}\label{H3A}
For any abelian group $A$ we have the exact sequence
\[
\begin{array}{c}
0 \arr \bigwedge_\z^3 A \arr H_3(A,\z) \arr \tors(A,A)^{\Sigma_2^\varepsilon} \arr 0,
\end{array}
\]
where $\Sigma_2^\varepsilon=\{\id,-\sigma\}$.
The homomorphism on the right side of the exact sequence is obtained from the composition
\[
H_3(A,\z) \overset{{\Delta}_\ast}{\larr } H_3(A\times A,\z) \arr \tors(A,A),
\]
where $\Delta$ is the diagonal map $A \arr A\times A$, $a \mt (a,a)$, and the action 
of $\sigma$ on  $ \tors(A,A)$ is induced by the involution $\iota: A \times A \arr A \times A$, $(a, b) \mt (b, a)$.
\end{prp}
\begin{proof}
 See \cite[Lemma~5.5]{suslin1991} or \cite[Section 6]{breen1999}. 
 \end{proof}
 
Here is our first main result.

\begin{thm}\label{2-torsion-image}
Let $A$ be a central subgroup of $G$ such that $A \subseteq G'$. Then the image of the natural map 
\[
H_3(A,\z)\arr H_3(G,\z)/\rho_\ast(A\otimes_\z H_2(G,\z))
\]
is $2$-torsion.
\end{thm}
\begin{proof}
By Proposition \ref{H3A} we have the exact sequence 
\[
\begin{array}{c}
0 \arr \bigwedge_\z^3 A \arr H_3(A,\z) \arr \tors(A,A)^{\Sigma_2^\varepsilon} \arr 0.
\end{array}
\]

From the diagram (\ref{prod}), we obtain the commutative diagram
\[
\begin{tikzcd}
\widetilde{H}_3(A\times A,\z)\ar[r, "\mu_\ast"]\ar[d]& H_3(A,\z)\ar[d]\\
\widetilde{H}_3(A\times G,\z)\ar[r, "\rho_\ast"]& H_3(G,\z),
\end{tikzcd}
\]
where
\[
{\widetilde{H}}_3(A\times A,\z):=\ker(H_3(A\times A,\z)\overset{({p_1}_\ast, {p_2}_\ast)}{-\!\!\!-\!\!\!-\!\!\!\larr}
H_3(A,\z) \oplus H_3(A,\z)),
\]
\[
\tilde{H}_3(A\times G,\z):=\ker(H_3(A\times G,\z)\overset{({p_1}_\ast, {p_2}_\ast)}{-\!\!\!-\!\!\!-\!\!\!\larr}
H_3(A,\z) \oplus H_3(G,\z)).
\]
As we have seen, the condition $A\subseteq G'$ implies that the composite 
\[
\begin{array}{c}
\bigwedge_\z^3 A\arr H_3(A,\z)\arr H_3(G,\z)
\end{array}
\]
is trivial. From this fact together with the K\"unneth formula \cite[Corollary~5.8, Chap.~V]{brown1994}
for ${\widetilde{H}}_3(A\times A,\z)$ we obtain the  commutative diagram
\[
\begin{tikzcd}
\tors(A,A) \ar[ddd, bend right=300] \ar[r, "\bar{\mu}_\ast"] & 
\tors(A,A)^{\Sigma_2^\varepsilon}\\
{\tilde{H}}_3(A\times A) /\bigoplus_{i=1}^2 H_i(A,\z)\otimes_\z H_{3-i}(A,\z)
\ar[ u, "\overset{\alpha}{\simeq}"]\ar[r, "\mu_\ast"] \ar[d, , "\widetilde{\inc}_\ast"]& 
H_3(A,\z)/\bigwedge_\z^3 A \ar[u, "\overset{\beta}{\simeq}"] \ar[d, "\inc_\ast"] \\
{\tilde{H}}_3(A\times G)/\bigoplus_{i=1}^2 H_i(A,\z)\otimes_\z H_{3-i}(G,\z) \ar[d, "\simeq"] \ar[r, "\rho_\ast"]& 
H_3(G,\z)/\rho_\ast(A\otimes_\z H_2(G,\z))\\
\tors(A,H_1(G,\z)). &
\end{tikzcd}
\]
Note that 
\[
\im(H_2(A,\z)\otimes_\z H_1(G,\z)\arr H_3(G,\z)) \se \im(A\otimes_\z H_2(G,\z)\arr H_3(G,\z))
\]
(see \cite[Proposition 4.4, Chap. V]{stammbach1973}).
Since the map 
\[
\tors(A,A)=\tors(A, H_1(A,\z)) \arr \tors(A,H_1(G,\z))
\]
is trivial, we see that $\rho_\ast\circ \widetilde{\inc}_\ast\circ \alpha^{-1}$ is trivial. This shows that 
$\inc_\ast\circ\beta^{-1}\circ \bar{\mu}_\ast$ is trivial. Therefore the image of $H_3(A,\z)$ in 
$H_3(G,\z)$ is equal to the image of
\[
\tors(A,A)^{\Sigma_2^\varepsilon}/\bar{\mu}_\ast\tors(A,A).
\]
By the above arguments, one sees that the homomorphism 
\[
\bar{\mu}_\ast:\tors(A,A)\arr \tors(A,A)^{\Sigma_2^\varepsilon}
\]
is induced by the composition $A\times A \overset{\mu}{\larr} A\overset{\Delta}{\larr} A\times A$.

If we apply Theorem \ref{E-H} to the central extension $A\tail A\two 1$, we obtain the exact sequence
\[
0\arr  \ker(\Psi) \arr H_4(K(A,2),\z) \overset{\Psi}{\arr} 
A\otimes_\z A \arr H_2(A,\z)\arr 0,
\]
where
\[
\ker(\Psi)\simeq H_3(A,\z)/\mu_\ast(A\otimes_\z H_2(A,\z) \oplus \tors(A,A)).
\]
Clearly $\mu_\ast(A\otimes_\z H_2(A,\z))= \bigwedge_\z^3 A\se H_3(A,\z)$.
Therefore
\[
\ker(\Psi)\simeq \tors(A,A)^{\Sigma_2^\varepsilon}/{(\Delta_A\circ\mu)}_\ast(\tors(A,A)).
\]
But in the previous section we have seen that $\ker(\Psi)$ is $2$-torsion. This proves the claim.

\end{proof}

\begin{cor}
 If $A\tail G\two Q$ is a universal central extension, then the image of $H_3(A,\z)$ 
in $H_3(G,\z)$ is $2$-torsion.
\end{cor}
\begin{proof}
Since the extension is universal, the homology groups $H_1(G,\z)$ and $H_2(G,\z)$ are trivial 
\cite[Corollary~4.1.18]{rosenberg1996}. Thus the claim follows from the previous theorem.
\end{proof}

\begin{rem}
(i) If $A$ is a central subgroup of a group $G$, then the same argument as in the proof of Theorem \ref{2-torsion-image}
shows that the image of the natural map
\[
H_3(A,\z)\arr H_3(G,\z)/\rho_\ast\widetilde{H}_3(A\times G,\z)
\]
is $2$-torsion.

\par (ii) In Proposition \ref{trivial} below, we show that if $A\tail G\two Q$ is a universal central
extension such that $K(Q,1)^+$ is an $H$-space, then the natural image of $H_3(A,\z)$ in  $H_3(G,\z)$
is trivial.
\end{rem}

\section{Third homology of central extensions over {\it H}-groups}\label{h-group}

For any sequence of abelian groups $A_n$, $n\geq 2$, Berrick and Miller constructed 
a perfect group $Q$ such that $H_n(Q,\z)\simeq A_n$ \cite[Theorem~1]{berrick-miller1992}.

Let $A$ be an abelian group. By the result of  Berrick and Miller, there is a perfect 
group $Q$ such that $H_2(Q,\z)\simeq A$ and $H_4(Q,\z)=0$. Now if $A \tail G \two Q$ is the 
universal central extension of $Q$ \cite[Theorem 4.1.3]{rosenberg1996}, then by Corollary \ref{E-H-1} we have the exact sequence
\[
0\arr H_4(K(A,2),\z) \arr H_3(G,\z)\arr H_3(Q,\z)\arr 0.
\]
This example shows that in general for an universal central extension $A \tail G \two Q$,
the kernel of $H_3(G,\z)\arr H_3(Q,\z)$ can be very complicated.

A group $Q$ is called {\it quasi-perfect} if its commutator subgroup is perfect, i.e.  $[Q',Q']=Q'$.  

A quasi-perfect  group $Q$ is called an $H$-{\it group} if $K(Q,1)^+$, the plus-construction of $K(Q,1)$ 
with respect to $Q'$ (see \cite[Section 1.1]{loday1976}), is an $H$-space. Note that 
for a group $G$, $K(G,1)$ is an $H$-space if and only if  $G$ is abelian.

\begin{exa}\label{directsum}
(a) A quasi-perfect group $Q$ with an internal ``direct sum'', that is a  homomorphism $Q\oplus Q \arr Q$,  is called a {\it direct sum group} if it 
satisfies in the following two conditions:
\par (i) for any $g_1, \dots, g_k\in Q'$ and any $g\in Q$, there is $h\in Q'$
such that $gg_ig^{-1}=hg_ih^{-1}$ for $1 \leq i \leq k$,
\par (ii) for any $g_1, \dots, g_n\in Q$, there are $c,d \in Q$ such that
$c(g_i\oplus 1)c^{-1}=d(1\oplus g_i)d^{-1}=g_i$.\\
By a theorem of Wagoner direct sum groups are $H$-groups \cite[Proposition 1.2]{wagoner1972}.

For a ring $R$ with unit, the stable general linear group  $\GL(R)$,  the stable orthogonal group $\OO(R)$, 
the stable symplectic groups $\Spp(R)$ (see the next section) are direct sum groups \cite[Example on page 323]{loday1976}.
For more examples and some details see \cite[Section 1.3]{loday1976}.
\par (b) For any abelian group $A$, Berrick has constructed a perfect group
${Q}$ such that $K(Q,1)^+$,  is homotopy equivalence to $K(A,2)$ 
\cite[Corollary 1.4]{berrick1987}. Since $K(A,2)$ is an $H$-space \cite[Theorem 7.11, Chap. V]{whitehead1978}, $Q$ is an $H$-group.
\end{exa}

The next theorem is our second main result.

\begin{thm}\label{A/2}
Let $A \tail G \two Q$ be a perfect central extension.
If $Q$ is an $H$-group, then we have the exact sequence
\[
A/2 \arr H_3(G,\z) /\rho_\ast(A\otimes_\z H_2(G,\z)) \arr H_3(Q,\z) \arr 0.
\]
\end{thm}
\begin{proof}
From the central extension and the fact that $Q$ is perfect we have 
the fibration 
\[
K(A, 1) \arr K(G,1)^+ \arr K(Q,1)^+
\]
\cite[Proposition 1]{wojtkowiak1985}, \cite[Theorem 6.4]{berrick1982}. From this 
we obtain the fibration 
\begin{equation}\label{fib+}
K(G,1)^+ \arr K(Q,1)^+ \arr K(A,2)
\end{equation}
\cite[Lemma 3.4.2]{may-ponto2012}. It is known that $K(A,2)$ is an $H$-space.
Moreover  the map $K(Q,1)^+ \arr K(A,2)$ is an
$H$-map \cite[Proposirion 2.3.1]{zabrodsky1976}. Since the plus-construction does not change the homology \cite[Theorem 1.1.1]{loday1976},
from the Serre spectral sequence of the fibration (\ref{fib+}) we obtain (see \cite[pages~347--349]{eckmann-hilton1971}) the exact sequence 
\[
H_4(Q,\z)\! \arr\! H_4(K(A,2),\z)\! \arr \!
H_3(G,\z)/\rho_\ast(A\otimes_\z H_2(G,\z))\! \arr\! H_3(Q,\z)\!\arr\! 0.
\]
From the commutative diagram, up to homotopy, of $H$-spaces and $H$-maps 
\[
\begin{tikzcd}
K(Q,1)^+ \times K(Q,1)^+ \ar[r] \ar[d] & K(Q,1)^+ \ar[d]\\
K(A,2) \times K(A,2) \ar[r] & K(A,2)
\end{tikzcd}
\]
we obtain the commutative diagram
\[
\begin{tikzcd}
H_2(Q,\z) \otimes_\z H_2(Q,\z) \ar[r] \ar[d] & H_4(Q,\z) \ar[d]\\
A \otimes_\z A \ar[r] & H_4(K(A,2),\z).
\end{tikzcd}
\]
Since $G$ is perfect, $H_2(Q,\z) \arr A$ is surjective.  This gives us the surjective map 
\[
H_4(K(A,2),\z)/\im (A\otimes_\z A)\two H_4(K(A,2),\z)/\im(H_4(Q,\z)).
\]
From this, Proposition \ref{KA2} and the exact sequence (\ref{exact-Gamma2}) we obtain the desired exact sequence.
\end{proof}

\begin{cor}\label{A/22}
Let $A \tail G \two Q$ be an universal central extension. If $Q$ is an $H$-group, then we have the exact sequence
\[
A/2 \arr H_3(G,\z) \arr H_3(Q,\z) \arr 0.
\]
\end{cor}

For a perfect central extension $A \tail G \two Q$, Theorem \ref{2-torsion-image}  implies that
the image of $H_3(A,\z)$ in $H_3(G,\z)/\rho_\ast(A\otimes_\z H_2(G,\z))$ is 2-torsion. In the next 
proposition we show that this image is trivial provided that $Q$ is an $H$-group. We need the following well-known fact.

\begin{lem}\label{ab}
For any abelian group $A$, we have the short exact sequence
\[
\begin{array}{c}
0\arr A/2 \overset{\bar{\psi}}{\larr} (A\otimes_\z A)_\sigma \overset{\beta}{\larr} \bigwedge_\z^2 A\arr 0,
\end{array}
\]
where $(A\otimes_\z A)_{\sigma}:=(A\otimes_\z A)/ \langle a\otimes b + b\otimes a | a,b \in A\rangle$,
$\bar{\psi}(\overline{a})=\overline{a\otimes a}$ and $\beta(\overline{a\otimes b})=a\wedge b$.
\end{lem}
\begin{proof}
This exact sequence is well-known and has appeared in \cite[page 70]{mikhailov2015} without a proof. But its proof
is rather classic. 

First note that $A\simeq\underset{i\in I}{\varinjlim} \ A_i$, where $\{A_i: i\in I\}$ is the direct system of all finitely generated subgroups of $A$. 
Now it is easy to see that 
\[
\begin{array}{c}
A\otimes_\z A\simeq \underset{i\in I}{\varinjlim} \ (A_i\otimes_\z A_i),  \ \ A/2\simeq\underset{i\in I}{\varinjlim}\  (A_i/2),  \  \ 
\bigwedge_\z^2A \simeq \underset{i\in I}{\varinjlim} \ \bigwedge_\z^2A_i
\end{array}
\]
(see \cite[6.3, Chap. V]{brown1994} for the third isomorphism). Furthermore, observe that 
\[
(A\otimes_\z A)_\sigma=H_0(\Sigma_2, A\otimes_\z A)={\rm Tor}_0^{\z[\Sigma_2]}(\z, A\otimes_\z A),
\] 
where $\Sigma_2=\{1, \sigma\}$ and
the action of $\sigma$ on $A\otimes_\z A$ is given by  $\sigma(a\otimes b)=b\otimes a$.
Since the direct limite (with direct system index) is an exact functor and commutes with the Tor-functor,
we may assume that $A$ is finitely generated. If $A=A_1\oplus A_2$, it is easy to see that the claim holds
for $A$ if and only if it holds for $A_1$ and $A_2$. Thus we can reduce the problem to cyclic groups.
Since for any cyclic group $A$,  $\bigwedge_\z^2A=0$ and $(A\otimes_\z A)_\sigma\simeq A/2$, the claim follows easily.
\end{proof}

\begin{prp}\label{trivial}
Let $A \tail G \two Q$ be a perfect central extension.
If $Q$ is an $H$-group, then the natural map
\[
H_3(A,\z) \arr H_3(G,\z) /\rho_\ast(A\otimes_\z H_2(G,\z)) 
\]
is trivial. In particular, if the extension is universal, then the natural map
$H_3(A,\z)\arr H_3(G,\z)$ is trivial.
\end{prp}
\begin{proof}
If we apply Theorem \ref{E-H} to the morphism of extensions
\[
\begin{tikzcd}
A \ar[r, tail, two heads]\ar[d]& A \ar[r]\ar[d] & \{1\} \ar[d]\\
A \ar[r, tail] &G \ar[r, two heads] & Q,
\end{tikzcd}
\]
we obtain the commutative diagram with exact rows
\[
\begin{tikzcd}
\ker(\Psi) \ar[r, "\simeq"] \ar[d] & \overline{H_3(A,\z)} \ar[d]                       &                    \\
H_4(K(A,2),\z) \ar[r]                           & H_3(G,\z) /\rho_\ast(A\otimes_\z H_2(G,\z))\!  \arr&\!\!\!\!\!\!\!\!\!\!\!\!\!\!\!\!\!\!\!H_3(Q,\z)\! \arr 0,
\end{tikzcd}
\]
 where $\overline{H_3(A,\z)}$ is a quotient of $H_3(A,\z)$  and the map 
 \[
 \Psi:\Gamma(A)=H_4(K(A,2),\z)\larr A\otimes_\z A
 \]
 is discussed in Section \ref{wqf}. Since $\Gamma(A)/\Omega(A)\simeq A/2$ (see the exact sequence (\ref{exact-Gamma2})), from  
 Theorem \ref{A/2} and the above diagram we obtain the commutative diagram
\begin{equation}\label{dd}
\begin{tikzcd}
\ker(\Psi) \ar[r, "\simeq"] \ar[d] & \overline{H_3(A,\z)} \ar[d] \\
A/2 \ar[r]                           & H_3(G,\z) /\rho_\ast(A\otimes_\z H_2(G,\z)).
\end{tikzcd}
\end{equation}
We have seen that the composite 
 \[
 A\otimes_\z A \overset{\Omega}{\larr} \Gamma(A) \overset{\Psi}{\larr} A\otimes_\z A,
 \]
 takes $a\otimes b$ to $a\otimes b+b\otimes a$. Thus from the commutative diagram
\[
\begin{tikzcd}
0\ar[r] &\ker(\Omega) \ar[r]\ar[d] & A\otimes_\z A \ar[r, "\Omega"]\ar[d, "\Omega"] & \im(\Omega) \ar[d, "\Psi"] \lar[r]&0\\
0\ar[r] &\ker(\Psi) \ar[r]& \Gamma(A)\ar[r, "\Psi"] &A\otimes_\z A & 
\end{tikzcd}
\]
and the exact sequence (\ref{exact-Gamma2}) we obtain the exact sequence
\[
\ker(\Psi) \arr A/2 \overset{\bar{\psi}}\arr (A\otimes_\z A)_{\sigma} \arr H_2(A,\z) \arr 0.
\]
By Lemma \ref{ab} the sequence 
\[
0\arr A/2 \overset{\bar{\psi}}\arr (A\otimes_\z A)_{\sigma} \arr H_2(A,\z) \arr 0
\]
is exact. Therefore the map $\ker(\Psi) \arr A/2$ is trivial. Now it follows from the
diagram (\ref{dd}) that the map $\overline{H_3(A,\z)}\arr H_3(G,\z) /\rho_\ast(A\otimes_\z H_2(G,\z))$ is trivial.
\end{proof}

\begin{exa}
Let $A \tail G \two Q$ be a perfect central extension and let $Q$ be an $H$-group. Here we would like to calculate the homomorphism 
\[
A/2 \arr H_3(G,\z)/\rho_\ast(A\otimes_\z H_2(G,\z))
\]
from Theorem \ref{A/2}.

The extension $A\tail G \two Q$ is an epimorphic image of the universal extension of $Q$, say $A_1\tail G_1 \two Q$, which is unique up 
to isomorphism. By Theorem \ref{A/2}, this gives us the commutative diagram
\[
\begin{tikzcd}
A_1/2 \ar[r] \ar[d, two heads] & H_3(G_1,\z)\ar[d]\\
A/2 \ar[r] & H_3(G,\z) /\rho_\ast(A\otimes_\z H_2(G,\z)).
\end{tikzcd}
\]
Thus we may assume that our extension is universal.

Thus let $A \tail G \two Q$ be an universal central extension such that $Q$ is an $H$-group.
From Corollary \ref{E-H-1} we have the exact sequence 
\[
H_4(Q,\z)\arr \Gamma(A)\arr H_3(G,\z) \arr H_3(Q,\z) \arr 0.
\]
By studying the Lyndon/Hochschild-Serre spectral sequences of the extension \cite[\S, Chap. VII]{brown1994}, we obtain the 
exact sequence 
\[
H_4(Q,\z)\arr B \arr H_3(G,\z)/ H_3(A,\z)\arr H_3(Q,\z) \arr 0,
\]
where 
\[
B=\ker(A\otimes_\z A \arr H_2(A,\z))=\lan a\otimes a: a\in A\ran. 
\]
By Proposition~\ref{trivial}, $H_3(G,\z)=H_3(G,\z)/H_3(A,\z)$. Now it
is easy to see that the following  diagram is commutative:
\[
\begin{tikzcd}
H_4(Q,\z)\ar[r]\ar[d, "="]&\Gamma(A)\ar[r]\ar[d, two heads, "\Psi"]&H_3(G,\z)\ar[d, "="]\ar[r]&H_3(Q,\z)\ar[d, "="]&\!\!\!\!\!\!\!\!\!\!\!\!\!\!\!\!\!\!\!\!\arr 0\\
H_4(Q,\z)\ar[r]& B \ar[r]& H_3(G,\z)\ar[r]& H_3(Q,\z)&\! \!\!\!\!\!\!\!\!\!\!\!\!\!\!\!\!\!\!\!\arr 0.
\end{tikzcd}
\]

Since the map $A/2 \arr H_3(G,\z)$ factors through $\Gamma(A)/H_4(Q,\z)$, it is also factors through
the group $B/H_4(Q,\z)$. In fact it factors throughout $B/\lan a\otimes b+b \otimes a|a,b\in A\ran$. 
Thus it is enough to calculate the map
\[
B/\lan a\otimes b+b \otimes a| a,b\in A\ran \overset{\eta}{ \larr} H_3(G,\z) .
\]

Let $Q=F/S$ be a free presentations of $Q$.
By a theorem of Hopf $H_2(Q,\z)\simeq (S \cap [F,F])/[S,F]$
\cite[Theorem 5.3, Chap. II]{brown1994}. This isomorphism can be given by  
the following explicit formula
\begin{align*}
\Lambda:(S \cap [F,F])/[S,F] &
\overset{\simeq}{\larr} H_2(Q,\z)=H_2(B_\bullet(Q)_Q),\\
\bigg(\prod_{i=1}^g [a_i,b_i]\bigg)[S,F] &\!\mapsto\!\!\!\!\!
\begin{array}{c}
{\sum_{i=1}^g 
\Big([\bar{s}_{i-1}|\bar{a}_i] \!+\![\bar{s}_{i-1}\bar{a}_i|\bar{b}_i]\!-\!
[\bar{s}_{i}\bar{b}_i|\bar{a}_i]-[\bar{s}_{i}|\bar{b}_i] \Big)},
\end{array}
\end{align*}
where ${s}_i=[{a}_1, {b}_1]\cdots [{a}_i, {b}_i]$
and for $x\in F$ we set $\bar{x}=xS\in F/S=Q$ 
\cite[Exercise~4, \S5, Chap. II]{brown1994}. Note that $\bar{s}_g=1$.
Here $B_\bullet(Q) \arr \z$ is the bar resolution of $Q$.

Let $G=F/R$, $Q=F/S$ and $A=S/F$ 
be free presentations of $G$, $Q$ and $A$, respectively.
Since $A$ is central we have $[S,F]\subseteq R$ and thus the following diagram
\[
\begin{array}{ccc}
H_2(Q,\z) \!\!\!& \overset{\simeq}
{-\!\!\!-\!\!\!-\!\!\!-\!\!\!-\!\!\!-\!\!\!-\!\!\!-\!\!\!-\!\!\!-\!\!\!-\!\!\!
-\!\!\!\larr}& \!\!\!\!\! A=S/R\\
{}_{\Lambda}\!\!\!\nwarrow\!\!\!\!\!\!\!\!\!\!\!\!\!\!\!\!\! & &\!\!\!
\hspace{-2.2cm} \nearrow\\ 
& (S\cap [F,F])/[S,F]. &
\end{array}
\]
commutes, where $(S \cap [F,F])/[S,F] \arr S/R=A$ is given by $s[S,F]\mapsto sR$. 
For any $a\in F$, we denote $aR\in G=F/R$ by $\hat{a}$ and for any 
$s\in S\cap [F,F]$, we denote $s[S,F]$ by $\tilde{s}$.

 The Lyndon-Hochschild-Serre spectral sequence 
\[
\E^2_{p,q}=H_p(Q ,H_q(A,\z))\Rightarrow H_{p+q}(G,\z)
\]
gives us a filtration of $H_3(G,\z)$
\[
0=F_{-1}H_3 \se  F_{0}H_3 \se F_{1}H_3 \se F_{2}H_3 \se F_{3}H_3=H_3(G,\z),
\]
such that $\E_{i,3-i}^\infty=F_{i}H_3/F_{i-1}H_3$. Now by an easy analysis of the above spectral 
sequence one sees that $F_{0}H_3 =F_{1}H_3=0$ and  the map $\eta$ is induced by the composite
\begin{equation}\label{maps}
B \arr E_{2,1}^3\simeq E_{2,1}^\infty \simeq F_2H_3
\subseteq H_3(G,\z).
\end{equation}
If $s_g=\prod_{i=1}^g[a_i,b_i]\in S\cap [F,F]$, then we need to compute 
\[
\eta(\Lambda(\tilde{s}_g)\otimes \hat{s}_g)\in H_3(G,\z)
\]
under the composition (\ref{maps}). By direct calculation, which we delete the details here,
this element maps to the following element of $H_3(G,\z)$:
\begin{gather*}
\hspace{-8.7cm}
\Lambda(s_g):=[\hat{s}_g|\hat{s}_g^{-1}|\hat{s}_g]+\\
\ \ \begin{array}{ll} 
\sum_{i=1}^g \Big(&
\!\!\!\!\![\hat{a}_i^{-1}|\hat{s}_{i-1}^{-1}|\hat{s}_g]
\!-\![\hat{a}_i^{-1}|\hat{s}_g|\hat{s}_{i-1}^{-1}]
\!-\![\hat{a}_i^{-1}|b_i^{-1}\hat{s}_{i}^{-1}|\hat{s}_g]
\!+\![\hat{a}_i^{-1}|\hat{s}_g|\hat{b}_i^{-1}\hat{s}_{i}^{-1}]+ \\
&
\!\!\!\!\![\hat{b}_i^{-1}|\hat{a}_i^{-1}\hat{s}_{i-1}^{-1}|\hat{s}_g]
\!-\![\hat{b}_i^{-1}|\hat{s}_g|\hat{a}_i^{-1}\hat{s}_{i-1}^{-1}]
\!-\![\hat{b}_i^{-1}|\hat{s}_{i}^{-1}|\hat{s}_g]
\!+\![\hat{b}_i^{-1}|\hat{s}_g|\hat{s}_{i}^{-1}]+\\
&
\!\!\!\!\![\hat{s}_g|\hat{a}_i^{-1}|\hat{s}_{i-1}^{-1}]
\!-\![\hat{s}_g|\hat{a}_i^{-1}|\hat{b}_i^{-1}\hat{s}_{i}^{-1}]
\!+\![\hat{s}_g|\hat{b}_i^{-1}|\hat{a}_i^{-1}\hat{s}_{i-1}^{-1}]
\!-\![\hat{s}_g|\hat{b}_i^{-1}|\hat{s}_{i}^{-1}]\Big).
\end{array}
\end{gather*}
\end{exa}

\section{The Hermitian $K$-theory}\label{k-th}

Let $Q$ be a perfect group. Then $K(Q,1)^+$ is a 1-connected CW-complex \cite[Theorem 5.2.2]{rosenberg1996} 
and by the theorem of Hurewicz
\[
\pi_2(K(Q,1)^+)\arr H_2(Q,\z)
\] 
is an isomorphism and
\[
\pi_3(K(Q,1)^+) \arr H_3(Q,\z)
\]
is surjective.

Let $A \tail G \two Q$ be a central extension with $Q$ perfect.
Then 
\[
K(A,1) \arr K(G,1)^+ \arr K(Q,1)^+
\]
is a fibration (see \cite[Corollary 8.4]{berrick1982} or \cite[Proposition 1]{wojtkowiak1985}), where the plus-construction
$K(G,1)^+$ is taken with respect to the maximal perfect subgroup of $G$. From this we obtain the exact sequence
\[
\cdots \arr \pi_3(K(A,1)) \arr \pi_3(K(G,1)^+) \arr \pi_3(K(Q,1)^+) \arr \pi_2(K(A,1)) \arr \cdots.
\]
Since $\pi_n(K(A,1))=0$ for $n\neq 1$, this implies that for $n\geq 3$,
\[
\pi_n(K(G,1)^+)\simeq \pi_n(K(Q,1)^+).
\]

\begin{lem}\label{pi3-h3}
Let $A\tail G \two Q$ be the universal central extension of the perfect group $Q$. Then $\pi_3(K(Q,1)^+)\simeq H_3(G,\z)$ and 
the Hurewicz map $\pi_3(K(Q,1)^+) \arr H_3(Q,\z)$ coincides with the natural map $H_3(G,\z) \arr H_3(Q,\z)$. 
\end{lem}
\begin{proof}
We proved in above that $\pi_3(K(G,1)^+) \simeq \pi_3(K(Q,1)^+)$. Since the extension
is universal,  $H_1(G,\z)=H_2(G,\z)=0$. Now Hurewicz's theorem implies that
\[
\pi_1(K(G,1)^+)=\pi_2(K(G,1)^+)=0.
\]
Thus $K(G,1)^+$ is 2-connected and again by Hurewicz's theorem 
\[
\pi_3(K(G,1)^+)\simeq H_3(G,\z).
\]
This implies that $\pi_3(K(Q,1)^+)\simeq H_3(G,\z)$. The other claim follows from the
commutative diagram 
\[
\begin{CD}
\pi_3(K(G,1)^+) @>\simeq>> H_3(G,\z)\\
@VV{\simeq}V             @VVV\\
\pi_3(K(Q,1)^+) @>>> H_3(Q,\z).
\end{CD}
\]
\end{proof}

\begin{prp}\label{exat-pi}
If $Q$ is an $H$-group, then we have the exact sequence
\[
\pi_2(K(Q,1)^+)/2 \arr \pi_3(K(Q,1)^+) \arr H_3(Q,\z)\arr 0.
\]
\end{prp}
\begin{proof}
Let $A\tail G \two Q$ be the universal central extension of $Q$. Since $\pi_2(K(Q,1)^+)\simeq H_2(Q,\z)\simeq A$,
the claim follows immediately from Proposition \ref{pi3-h3} and Corollary \ref{A/22}.
\end{proof}

Let $R$ be an associative ring with unit. Let there be an involution on $R$, that is an automorphism of the
additive group of $R$, $R \arr R$ with $r \mt \overline{r}$, such that $\overline{\overline{r}}=r$ and
$\overline{rs}=\overline{s}\ \overline{r}$. 

Let $\epsilon$ be an element in the center of $R$ such that $\epsilon\overline{\epsilon}=1$. Set 
\[
\text{$R_\epsilon :=\{ r-\epsilon\overline{r} : r \in R\}$ \ \ \ and \ \ \  $R^\epsilon :=\{ r \in R: \epsilon\overline{r}=-r \}$}
\]
and observe that $R_\epsilon \se R^\epsilon$.  Let 
\[
M_n(R)_\epsilon:=\{ A-\epsilon  \overline{A}^T: A\in M_n(R) \},
\]
where $\overline{(a_{ij})}=(\overline{a_{ij}})$.
Let $e_{i, j}(r)$ be the $2n \times 2n$-matrix with $r \in R$ in the $(i, j)$-place and zero elsewhere. 

We define the {\it unitary and orthogonal group} of the pair $(R,\epsilon)$ as follow
\[
\!\!\!\!\!\!\!\!\!\!\!\!\!\!\!\!\!\!\!\!\!\!\!\!\!\!\!
{}_\epsilon \U_{2n}(R) :=\{ A \in \GL_{2n}(R) : \overline{A}^TF_nA=F_n \},
\]
\[
{}_\epsilon \OO_{2n}(R) :=\{ A \in \GL_{2n}(R) : \overline{A}^TQ_nA-Q_n\in M_n(R)_\epsilon \},
\]
respectively, where
\[
F_n= \sum_{i=1}^n(e_{2i-1, 2i}(1)+e_{2i, 2i-1}(\ep)), \ \ \ Q_n:= \sum_{i=1}^ne_{2i-1, 2i}(1).
\]
We have always 
\[
{}_\epsilon \OO_{2n}(R)\se {}_\epsilon \U_{2n}(R).
\]
If $R_\epsilon=R^\epsilon$, then ${}_\epsilon \U_{2n}(R)={}_\epsilon \OO_{2n}(R)$. Observe that 
if there is an element $s$ in the center of $R$ such that $s+\overline{s} \in R^\times$, in particular if $2 \in R^\times$, then $R_\epsilon=R^\epsilon$. 
\begin{exa}
(i) Let $\epsilon=-1$ and let the involution be the identity map ${\rm id}_R$.
Then 
\[
{}_\epsilon \U_{2n}(R) =\Spp_{2n}(R)
\]
is the usual symplectic group. Note that $R$ is commutative in this case.
\par (ii) Let $\epsilon=1$ and let the involution be the identity map ${\rm id}_R$. Then
\[
{}_\epsilon \OO_{2n}(R)  =\OO_{2n}(R)
\]
is the usual orthogonal group. As in the symplectic case, $R$ is necessarily commutative.
\par (iii) Let $\epsilon=-1$ and let the involution is not the identity map ${\rm id}_R$. Then 
\[
{}_\epsilon \U_{2n} (R) = \U_{2n}(R)
\]
is the classical unitary group corresponding to the involution.
\end{exa}

Let $\alpha$ be the permutation of the set of natural numbers given by $\alpha(2i)=2i-1$ and  $\alpha(2i-1)=2i$. For $1 \le i, j \le 2n$, $i \neq j$, 
and every $r \in R$ define
\begin{gather*}
E_{i, j}(r)=
\begin{cases}
I_{2n}+ e_{i, j}(r) & \text{if $i=2k-1, j\!=\!\alpha(i)$, $r\!=\!-\overline{\epsilon}\ \overline{r}$}\\
I_{2n}+ e_{i, j}(r) & \text{if $i=2k, j=\alpha(i)$, $r=-\epsilon\ \overline{r}$}\\
I_{2n}+ e_{i, j}(r)+ e_{\alpha(j), \alpha(i)}(-\overline{r})&
\text{if $i+j=2k$, $i\neq j$}\\
I_{2n}+ e_{i, j}(r)+ e_{\alpha(j), \alpha(i)}(-\epsilon^{-1}\overline{r})&
\text{if $i \neq\alpha(j)$, $i=2k-1$, $j=2l$}\\
I_{2n}+ e_{i, j}(r)+ e_{\alpha(j), \alpha(i)}(\epsilon\overline{r})&
\text{if $i \neq\alpha(j)$, $i=2k$, $j=2l-1$} \end{cases}
\end{gather*}
where $I_{2n}$ is the identity element of $\GL_{2n}(R)$. It is easy to see that $E_{i, j}(r) \in {}_\epsilon \U_{2n}(R)$. 

Let  ${}_\epsilon \EU_{2n}(R)$ be the subgroup of ${}_\epsilon \U_{2n}(R)$ generated by the matrices $E_{i, j}(r)$, $r \in R$, and 
${}_\epsilon \EO_{2n}(R)$ be the group of ${}_\epsilon \OO_{2n}(R)$ generated by the matrices $E_{i, j}(r)$, $r \in R$, which are in 
${}_\epsilon \OO_{2n}(R)$.  We call ${}_\epsilon \EU_{2n}(R)$ the  {\it elementary unitary group} and ${}_\epsilon \EO_{2n}(R)$ the  
{\it elementary orthogonal group}.

The natural embeddings
\[
{}_\epsilon \OO_{2n}(R)  \arr {}_\epsilon \OO_{2(n+1)}(R), \ \ \ {}_\epsilon \U_{2n}(R)  \arr {}_\epsilon \U_{2(n+1)}(R),
\]
given by 
$ A \mt
\left(\begin{array}{cc}
A & 0      \\
0 &  I_2
\end{array} \right)$,
embeds ${}_\epsilon \EO_{2n}(R) $ and ${}_\epsilon \EU_{2n}(R)$ naturally in ${}_\epsilon \EO_{2(n+1)}(R)$ and ${}_\epsilon \EU_{2(n+1)}(R)$,
respectively. Define the {\it stable unitary and ortogonal groups} and their {\it stable elementary subgroups} as follow:
\[
\begin{array}{c}
{}_\epsilon \OO(R)  :=\bigcup_{n\geq 1}{}_\epsilon \OO_{2n}(R), \ \ \ \ \ {}_\epsilon \EO(R)  :=\bigcup_{n\geq 1}{}_\epsilon \EO_{2n}(R),
\end{array}
\]
\[
\begin{array}{c}
{}_\epsilon \U(R)  :=\bigcup_{n\geq 1}{}_\epsilon \U_{2n}(R), \ \ \ \ \ {}_\epsilon \EU(R)  :=\bigcup_{n\geq 1}{}_\epsilon \EU_{2n}(R).
\end{array}
\]
It is a classical result that ${}_\epsilon \EO(R)$ (resp. ${}_\epsilon \EU(R)$)  is perfect and is  normal in  ${}_\epsilon \OO(R)$ (resp. ${}_\epsilon \U(R)$). 
More precisely,
\[
{}_\epsilon \EO(R) =[{}_\epsilon \EO(R), {}_\epsilon \EO(R)]
=[{}_\epsilon \OO(R), {}_\epsilon \OO(R)],
\]
\[
{}_\epsilon \EU(R) =[{}_\epsilon \EU(R), {}_\epsilon \EU(R)]
=[{}_\epsilon \U(R), {}_\epsilon \U(R)],
\]
(see \cite[Theorems 1.4 and 1.4$_o$]{vaserstein1970}). For $n\geq 1$, the  {\it Hermitian and Unitary $K$-groups of the 
pair $(R, \epsilon)$} are defined as follow
\[
{}_\epsilon K_n^h(R):=\pi_n(K({}_\epsilon \OO(R),1)^+), 
\]
\[
{}_\epsilon K_n^u(R):=\pi_n(K({}_\epsilon \U(R),1)^+),
\]
where the plus-constructions are taken with respect to the perfect subgroups ${}_\epsilon \EO(R)$ and ${}_\epsilon \EU(R)$, respectively. Since
${}_\epsilon \OO(R)$ and ${}_\epsilon \U(R)$ are quasi-perfect, for $n\geq 2$ we have
\[
{}_\epsilon K_n^h(R)\simeq\pi_n(K({}_\epsilon \EO(R),1)^+),
\]
\[
 {}_\epsilon K_n^u(R)\simeq\pi_n(K({}_\epsilon \EU(R),1)^+).
\]

One can show that the homomorphism
\[
(A, B)\mt A\oplus B,
\]
where 
\[
(A\oplus B)_{ij}:=
\begin{cases}
A_{kl} & \text{if $i=2k-1$, $j=2l-1$}     \\
B_{kl} & \text{if $i=2k$, $j=2l$}     \\
0 & \text{otherwise}
\end{cases}
\]
defined on ${}_\epsilon \OO(R)$ and ${}_\epsilon \U(R)$
satisfies in conditions (i) and (ii) of Example \ref{directsum}(a) (\cite[page 324, Example (d)]{loday1976}). This
implies that ${}_\epsilon \OO(R)$ and ${}_\epsilon \U(R)$ are $H$-groups.
Now by Proposition \ref{exat-pi} we obtain the following result.

\begin{thm}
Let $R$ be a ring with an involution and a central element $\ep$ such that $\ep\overline{\ep}=1$.
Then we have the exact sequences
\[
{}_\epsilon K_2^h(R)/2\arr {}_\epsilon K_3^h(R)\arr H_3({}_\epsilon \EO(R),\z)\arr 0,
\]
\[
{}_\epsilon K_2^u(R)/2\arr {}_\epsilon K_3^u(R)\arr H_3({}_\epsilon \EU(R),\z)\arr 0.
\]
\end{thm}

With a similar argument we can prove a similar result for Algebraic $K$-groups of a ring: For any ring $R$ with unit we have the exact sequence
\[
K_2(R)/2\arr K_3(R)\arr H_3(\textrm{E}(R),\z)\arr 0,
\]
where $\textrm{E}(R)$ is the elementary subgroup of the stable linear group $\GL(R)$.
This exact sequence already was proved, with a different method, by Suslin in \cite[Corollary 5.2]{suslin1991}. 
We should mention that Suslin also studied the nature of the map $K_2(R)/2\arr K_3(R)$.

\section{The third cohomology of central perfect extensions}\label{coho}

Let $A \tail G \two Q$ be a perfect central extension. From the fibration
\[
K(G,1)^+ \arr K(Q,1)^+ \arr K(A,2),
\]
(see the proof of Theorem \ref{A/2}) we obtain the Serre spectral sequence
\[
E_2^{p,q}=H^p(K(A,2), H^q(G,\z)) \Rightarrow H^{p+q}(Q,\z),
\]
(see \cite[Section~9.5]{davis-kirk2001}).
By the universal coefficients theorem for the cohomology of groups \cite[Exercise 3, Section 1, Chap. III]{brown1994} and 
spaces \cite[Corollary 2.35]{davis-kirk2001} we have 
\[
H^1(G,\z)=0, \ \ H^2(G,\z)\simeq \Hom(H_2(G,\z),\z), \ \ H^1(K(A,2),\z)=0, 
\]
\[
H^2(K(A,2),\z)\simeq \Hom(A,\z), \ \ \ H^3(K(A,2),\z)\simeq \exts(A,\z) 
\]
and
\[  
H^4(K(A,2),\z)\simeq \Hom(\Gamma(A),\z).
\]
For any $n\geq 0$, the spectral sequence gives us a filtration of $H^n(Q,\z)$ as follow
\[
0=F^{n+1}H^n\se F^nH^n\se \cdots\se F^1H^n\se F^0H^n=H^n(Q,\z)
\]
such that 
\[
E_\infty^{p,n-p}\simeq F^pH^n/F^{p+1}H^n. 
\]
By the universal coefficients theorems for groups and spaces and the above calculations, we obtain
\[
E_2^{p,1}=0,  \ \ \ E_2^{0,q}\simeq H^q(G,\z),  \ \ \ \ p,q\geq 0,
\]
and
\[
E_2^{2,0}\simeq \Hom(A,\z), \ \ \ E_2^{3,0}\simeq \exts(A,\z), \ \ \ E_2^{1,2}=0, 
\]
\[
E_2^{2,2}\simeq \Hom(A, H^2(G,\z)), \ \ \ E_2^{4,0}\simeq \Gamma(A)^\ast,
\]
where for an abelian group $M$, $M^\ast$ is the dual group $\Hom_\z(M,\z)$. Now by a
direct analysis of the filtrations of $H^2(Q,\z)$ and  $H^3(Q,\z)$ we obtain the exact sequence
\begin{equation}\label{coh-ex}
0\arr \Hom(A,\z)\! \arr H^2(Q,\z) \arr H^2(G,\z)\! \arr
\exts(A,\z) \arr H^3(Q,\z)
\end{equation}
\[
\arr \ker\Big(H^3(G,\z) \arr \Hom(A, H^2(G,\z))\Big)
\arr \Gamma(A)^\ast.
\]
By Theorem \ref{E-H}, we have the exact sequence 
\[
0 \arr H_2(G,\z)\arr H_2(Q,\z)\arr A\arr 0.
\]
This together with the isomorphisms $H^2(G,\z)\simeq \Hom(H_2(G,\z),\z)$ and $H^2(Q,\z)\simeq \Hom(H_2(Q,\z),\z)$ imply that 
the natural map $H^2(Q,\z) \arr H^2(G,\z)$ is surjective.
Thus from the exact sequence (\ref{coh-ex}) and the isomorphism
\begin{align*}
\Hom(A, H^2(G,\z)) & \simeq \Hom(A, \Hom(H_2(G,\z),\z))\\
& \simeq (A\otimes_\z H_2(G,\z))^\ast,
\end{align*}
we obtain the exact sequence
\[
0\!\arr\! \exts(A,\z) \!\arr\! H^3(Q,\z) \!\arr \!
\ker\Big(H^3(G,\z)\! \overset{\rho^\ast}{\arr} \!(A\otimes_\z H_2(G,\z))^\ast\Big)\!\arr\! \Gamma(A)^\ast.
\]
This proves the first part of the following proposition.

\begin{prp}
Let $A \tail G \two Q$ be a perfect central extension. Then we have the exact sequence
\[
0\!\arr\! \exts(A,\z) \!\arr\! H^3(Q,\z) \!\arr \!
\ker\Big(H^3(G,\z)\! \overset{\rho^\ast}{\arr} \!(A\otimes_\z H_2(G,\z))^\ast\Big)\!\arr\! \Gamma(A)^\ast.
\]
In particular, if $Q$ is an $H$-group, then we have the exact sequence
\[
0\arr \exts(A,\z) \arr H^3(Q,\z) \arr H^3(G,\z) \overset{ \rho^\ast}\arr (A\otimes_\z H_2(G,\z))^\ast.
\]
\end{prp}

\begin{proof}
We already proved the first part. So let $Q$ be an $H$-group.
First assume that the extension is universal. Then the above extension finds the following form
\[
0\arr \exts(A,\z) \arr H^3(Q,\z) \arr H^3(G,\z) \arr \Gamma(A)^\ast.
\]
From the proof of Theorem \ref{A/2} we see that the map $\Gamma(A)\arr H_3(G,\z)$ factors
throught $A/2=\Gamma(A)/\Omega(A)$. Thus 
\[
H^3(G, \z)\simeq H_3(G,\z)^\ast \arr \Gamma(A)^\ast
\]
factors through $(A/2)^\ast=0$, which implies that it is trivial.

In general the extension $A \tail G \two Q$ is an epimorphic image of a universal central extension of $Q$,
say $A_1 \tail G_1 \two Q$. Thus $A_1\simeq H_2(Q,\z)$ and  we have a morphism of extensions
\[
\begin{tikzcd}
A_1 \ar[r, tail]\ar[d, two heads]& G_1 \ar[r, two heads]\ar[d, two heads] & Q\ar[d, "="]\\
A \ar[r, tail] & G \ar[r, two heads] & Q.
\end{tikzcd}
\]
This gives us the commutative diagram of exact sequences
\[
\begin{tikzcd}
0\!\ar[r] & \exts(A,\z) \ar[r] \ar[d] &\! H^3(Q,\z) \ar[r]\ar[d, "="]&\! \widetilde{H}^3(G,\z) \ar[r]\ar[d] &\!\Gamma(A)^\ast\ar[d, hook]\\
0\!\ar[r] & \exts(A_1,\z) \ar[r] &\! H^3(Q,\z) \ar[r] &\! H^3(G_1,\z) \ar[r, "0"] &\! \Gamma(A_1)^\ast.
\end{tikzcd}
\] 
where 
\[
\widetilde{H}^3(G,\z):=\ker\Big(H^3(G,\z) \overset{\rho^\ast}\larr (A\otimes_\z H_2(G,\z))^\ast\Big).
\]
(Note that since $A_1\arr A$ is surjective, $\Gamma(A_1)\arr \Gamma(A)$ is surjective too. This implies that the map
$\Gamma(A)^\ast \arr \Gamma(A_1)^\ast$ is injective.) Now from the above diagram we see that
the map  $\widetilde{H}^3(G,\z) \arr\Gamma(A)^\ast$ is trivial. This proves our claim.
\end{proof}

Note that the second part of the above proposition is the cohomology analogue of Theorem \ref{A/2}.

\begin{cor}
Let $A \tail G \two Q$ be an universal central extension. Then we have the exact sequence
\[
0\arr \exts(A,\z) \arr H^3(Q,\z) \arr H^3(G,\z) \arr \Gamma(A)^\ast.
\]
In particular, if $Q$ is an $H$-group, then we have the exact sequence 
\[
0\arr \exts(A,\z) \arr H^3(Q,\z) \arr H^3(G,\z) \arr 0.
\]
\end{cor}

\begin{rem}
Let $A$ be a central subgroup of $G$ and let $A\se G'$. Let $i:A\arr G$ be the usual inclusion map. We have seen at the beginning 
of Section \ref{stem}, that $i_\ast: H_2(A,\z)\arr H_2(G,\z)$ is trivial and the image of $i_\ast: H_3(A,\z)\arr H_3(G,\z)$ is torsion. Thus 

\[
i^\ast: H_3(G,\z)^\ast\arr H_3(A,\z)^\ast
\]
is trivial (because it factors through $\Hom( \im(i_\ast),\z)=0$). Now  it follows from the 
commutative diagram
\[
\begin{tikzcd}
0\! \ar[r] &\! \exts(H_2(G,\z),\z)\ar[r]\ar[d, "0"] & H^3(G,\z)\!\ar[r]\ar[d] & H_3(G,\z)^\ast \!\ar[r]\ar[d, "0"]&\! 0\\
0\! \ar[r] &\! \exts(H_2(A,\z),\z)\ar[r]          & H^3(A,\z)\!\ar[r]           & H_3(A,\z)^\ast\! \ar[r]        &\! 0,
\end{tikzcd}
\]
that 
\[
\begin{array}{c}
\im(H^3(G,\z) \arr H^3(A,\z))\se \exts(\bigwedge_\z^2 A,\z). 
\end{array}
\]
\end{rem}


\end{document}